%% file: pseudo-arithmetic.tex
\pdfoutput=1
\documentclass[a4paper,11pt]{amsart}

\usepackage{hyperref}

\usepackage{mathptmx}
\usepackage{amsmath}
\usepackage{amssymb}
\usepackage{amsthm}
\usepackage{array}
\usepackage[all,tips,line]{xy}
\usepackage{enumerate}
\usepackage{graphicx}
\usepackage{lmodern}
\usepackage{mathrsfs}
\usepackage[utf8]{inputenc}
\usepackage{multirow}

\usepackage{tikz-cd}

\usepackage{subcaption}
\hypersetup{colorlinks,citecolor=blue,linkcolor=blue}

\newtheorem{theorem}{Theorem}[section]
\newtheorem{lemma}[theorem]{Lemma}
\newtheorem{prop}[theorem]{Proposition}
\newtheorem{cor}[theorem]{Corollary}

\theoremstyle{definition}
\newtheorem{definition}[theorem]{Definition}
\newtheorem{question}[theorem]{Question}
\newtheorem{example}[theorem]{Example}
\newtheorem{rmk}[theorem]{Remark}

\numberwithin{equation}{section}

\input{commands}
\begin{document}
\title{Hyperbolic manifolds and pseudo-arithmeticity}

\begin{abstract} We introduce and motivate a notion of pseudo-arithmeticity,
        which possibly applies to all lattices in $\PO(n,1)$ with $n>3$. We further show that under
        an additional assumption (satisfied in all known cases), the covolumes of these lattices correspond to rational linear
        combinations of special values of $L$-functions.
\end{abstract}

\author{Vincent Emery and Olivier Mila}
\thanks{This work is supported by the Swiss National Science Foundation, Project number PP00P2\_157583}

\address{
Universit\"at Bern\\
Mathematisches Institut\\
Sidlerstrasse 5\\
CH-3012 Bern\\
Switzerland
}
\email{vincent.emery@math.ch}
\email{olivier.mila@math.unibe.ch}

\date{\today}


\maketitle

\input{intro}
\input{glueings}

\input{homology}
\input{Coxeter}

\clearpage
\bibliographystyle{amsplain}
\bibliography{all}

\end{document}

%% file: commands.tex
\DeclareMathOperator{\GL}{GL}

\DeclareMathOperator{\PO}{PO}

\DeclareMathOperator{\PU}{PU}

\DeclareMathOperator{\PGL}{PGL}

\DeclareMathOperator{\conj}{conj}
\DeclareMathOperator{\Stab}{Stab}

\DeclareMathOperator{\vol}{vol}

\newcommand{\im}{\mathrm{im}}

\newcommand{\POO}{\mathbf{PO}}
\newcommand{\OO}{\mathbf{O}}
\newcommand{\POp}{\POO^+}
\newcommand{\GO}{\mathrm{GO}}

\newcommand{\PV}{\mathbf{P}V}

\newcommand{\Z}{\mathbb Z}
\newcommand{\Q}{\mathbb Q}
\newcommand{\R}{\mathbb R}
\newcommand{\C}{\mathbb C}

\renewcommand{\O}{\mathcal O}

\newcommand{\La}{\mathcal{L}}

\newcommand{\V}{V}

\newcommand{\Hy}{\mathbf{H}}

\renewcommand{\P}{\mathbf P}
\newcommand{\G}{\mathbf G}

\newcommand{\sG}{\widetilde{\mathbf{G}}}

\newcommand{\ZZ}{\mathbf Z}

\newcommand{\bs}{\backslash}

\newcommand{\Ad}{\mathrm{Ad}}
\newcommand{\g}{\mathfrak{g}}

\newcommand{\Gal}{\mathrm{Gal}}

 \newcommand{\Oplus}{\bigoplus}

\newcommand{\longto}{\longrightarrow}   
\newcommand{\inj}{\hookrightarrow}     

\newcommand{\surj}{\twoheadrightarrow}

\newcommand{\tM}{\widetilde M}

\newcommand{\aPO}{\mathbf{PO}}

\newcommand{\diag}{\mathrm{diag}}

%% file: intro.tex
\section{Introduction}
\label{sec:intro}

The study of hyperbolic $3$-manifolds of finite volume has many relations with number theory, with a
central role in this context being played by the notion of the \emph{invariant trace field} (see
\cite{MaclReid03}).  The work of Vinberg \cite{Vinb71} allows to define a comparable invariant (the
\emph{adjoint trace field}) for any locally symmetric space. This paper studies finite volume
hyperbolic manifolds of dimensions $n>3$ from the perspective of the adjoint trace field, and the
algebraic groups that are naturally associated with them (see Theorem~\ref{thm:Vinberg}).  

Our main purpose is to introduce and motivate a notion of ``pseudo-arithmeticity'', which to the best
of our knowledge applies to all {\em currently known} lattices in $\PO(n,1)$ with $n>3$. In this
introduction we state two main results: Theorem~\ref{thm:k-G-glueings}, where we show that the
classical hyperbolic manifolds obtained by glueing are pseudo-arithmetic; and
Theorem~\ref{thm:vol-odd}, in which it is proved that under some mild assumption (defined in
Sect.~\ref{sec:first-type}) the covolumes of pseudo-arithmetic lattices correspond to rational
linear combinations of covolumes of arithmetic lattices.

\subsection{Hyperbolic isometries and algebraic groups}
\label{sec:hyp-isom}

We denote by $\Hy^n$ the hyperbolic $n$-space, and we identify its group of
isometries with the Lie group  $\PO(n,1)$.  Let  $\POO$ denote the
real algebraic group such that $\POO(\R) = \PO(n,1)$, and $\POp$ its identity component.
Note the following dichotomy:
\begin{itemize}
        \item for $n$ even $\POO$ is connected, so that $\POp(\R) = \PO(n,1)$;
        \item for $n$ odd  $\POO$ has two connected components, and we have $\POp(\R) = \PO(n,1)^\circ$. 
\end{itemize}

Recall that any complete hyperbolic $n$-manifold $M$ can be written as a quotient $M = \Gamma\bs\Hy^n$, where
$\Gamma \subset \PO(n,1)$ is a torsion-free discrete subgroup (uniquely defined up to conjugacy).
The manifold $M$ has finite volume exactly when $\Gamma$ is a lattice in $\PO(n,1)$. We will usually assume that
$M$ is orientable, i.e., $\Gamma \subset \PO(n,1)^\circ$. Then by Borel's density theorem $\Gamma$ is Zariski-dense in
$\POp$.

Let $K/k$ be a field extension; for an
algebraic $k$-group $\G$ the symbol $\G_K$ denotes the $K$-group obtained by scalar extension (base
change induced by $k \subset K$). In this case $\G$ is said to be a \emph{$k$-form} of $\G_K$.
Assume that  $k$ is a number field, with ring of integers $\O_k$. Then a $k$-form $\G$ of $\POp$ (or of
$\POO$) is called \emph{admissible} (for $\PO(n,1)$) if the Lie group $\G(k\otimes_\Q \R)^\circ$
contains exactly one noncompact factor, isomorphic to $\PO(n,1)^\circ$. In this case $\G(\O_k)$ is 
a lattice in $\PO(n,1)$.

\begin{rmk}
       \label{rmk:admiss-tot-real}
       For $n>3$ the admissibility condition implies that $k$ is totally real.   
\end{rmk}

\subsection{Trace fields and ambient groups}
\label{sec:trace-ambient}

For any subgroup $\Gamma \subset \PO(n,1)$ with $n>3$, we define its \emph{(adjoint) trace field} as 
the subfield of $\R$ given by
\begin{align}
        k &= \Q(\left\{  \mathrm{tr}(\Ad \gamma) | \gamma \in  \Gamma)\right\}),
        \label{eq:adj-trace-field}
\end{align}
where $\Ad: \PO(n,1) \to \GL(\g)$ is the adjoint representation. In case $\Gamma$ is a lattice,
it follows from Weil's local rigidity that $k$ is a number field (see \cite[Prop.~1.6.5]{OniVinb-II}). 
The work of Vinberg \cite{Vinb71} shows the following.
\begin{theorem}[Vinberg]
        Let $\Gamma$ be a Zariski-dense subgroup of $\POp(\R)$ (resp.\ of $\POO(\R)$),
        with trace field $k$. Then
        \begin{enumerate}
                \item $k$ is an invariant of the commensurability class of $\Gamma$; 
                \item there exists a $k$-form $\G$ of $\POp$ (resp.\ of $\POO$) such that 
                        $\Gamma \subset \G(k)$.
        \end{enumerate}
        \label{thm:Vinberg}
\end{theorem}

Since $\Gamma$ is Zariski-dense, the group $\G$ is uniquely determined by $\Gamma$ (up to
$k$-isomorphism), and it is a commensurability invariant.  We call it the \emph{ambient group} of $\Gamma$.
If $\G$ is admissible then $\Gamma$ is called \emph{quasi-arithmetic}.
If moreover $\Gamma$ is commensurable with $\G(\O_k)$, it is called \emph{arithmetic}.
We use the same terminology for the corresponding quotient $\Gamma\bs\Hy^n$.

\begin{rmk}
       \label{rmk:field-of-def}
       This definition of arithmetic lattices differs from the usual one (see for instance
       \cite[Sect.~3.6]{OniVinb-II}); but it is
       equivalent to it. The equivalence essentially follows from the fact that trace fields of arithmetic lattices
       coincide with their fields of definition (see \cite[Lemma~2.6]{PraRap09}).
\end{rmk}

\begin{rmk}
        \label{rmk:dim-3-trace-field}
       For $n=3$ the group $\PO(3,1)^\circ$ is isomorphic to $G = \PGL_2(\C)$ and thus carries a complex
       structure. Taking the \emph{complex} adjoint representation $G \to \GL(\g_\C)$ instead in the
       definition, the adjoint trace field coincides with the ``invariant trace field'' (see
       \cite[Exercises~3.3~(4)]{MaclReid03}).
\end{rmk}

\subsection{The case of glued manifolds}
\label{sec:intro-glued-manif}

To date for $n>3$ there are two sources of nonarithmetic lattices in $\PO(n,1)$:
\begin{enumerate}
        \item Hyperbolic reflection groups\footnote{Note that hyperbolic reflections groups (of finite covolume) cannot 
                        exist for $n\ge 996$, and no examples are known for $n > 21$ (see \cite[Sect.~1]{Bel16}).}, some of which can be proved to be
                nonarithmetic by using Vinberg's criterion \cite{Vinb67}.
        \item Hyperbolic manifolds constructed by glueing together ``pieces'' of arithmetic manifolds  along pairwise isometric totally geodesic hypersurfaces.
\end{enumerate}
There are several constructions that fit the description in (2) (see \cite{Raim12,GelLev14,BelThom11}), all of which can be thought as 
(clever) variations of the original method by Gromov and Piateski-Shapiro \cite{GroPS87}.  
Their building blocks are always \emph{arithmetic pieces} $M_i$, i.e.,
hyperbolic $n$-manifolds with totally geodesic boundary taken inside arithmetic manifolds
$\Gamma_i\bs\Hy^n$ (see Sect.~\ref{sec:hyper-pieces}--\ref{sec:arithm-pieces}).  We will see in Corollary
\ref{cor:glue-same-k} that two such pieces $M_i$ and $M_j$ cannot be glued together
unless $\Gamma_i$ and $\Gamma_j$ have the same field of definition (equivalently, the same trace field). 

The following theorem applies to all nonarithmetic lattices constructed in the sense of (2). By
a \emph{multiquadratic} extension $K/k$ we mean a  (possibly trivial)  field extension of the form
$K = k(\sqrt{a_1},\dots, \sqrt{a_s})$ ($a_i \in k$, $s\ge 0$).

\begin{theorem}
        For $n>3$, let $M = \Gamma\bs\Hy^n$ be a hyperbolic $n$-manifold  obtained by glueing together a finite number of
        arithmetic pieces $M_i \subset \Gamma_i\bs\Hy^n$, each of the arithmetic lattices $\Gamma_i$ being defined over $k$.
        Then: 
        \begin{enumerate}
                \item The trace field $K$ of $\Gamma$ is a multiquadratic extension of $k$.
                \item The ambient group of $\Gamma$ is $\G_K$, where $\G$ is the ambient group of any
                        of the arithmetic lattices $\Gamma_i$. 
        \end{enumerate}
        \label{thm:k-G-glueings}
\end{theorem}

The exact description of the trace field $K$ -- and in particular the degree of $K/k$ -- depends on
the pieces $M_i$ and the way they are glued together. A precise treatment is the subject of a
separate article by the second author \cite{Mila}.


\subsection{Pseudo-arithmetic lattices}
\label{sec:intro-pseudo-arithm}

Theorem~\ref{thm:k-G-glueings} motivates the following definitions:

\begin{definition}[pseudo-admissible]
        We call \emph{pseudo-admissible} (over $K/k$) an algebraic group of the form $\G_K$, where $K/k$ is
        a finite (possibly trivial) real multiquadratic extension and $\G$ is an admissible $k$-group. 
\end{definition}

\begin{definition}[pseudo-arithmetic]
        \label{def:pseudo-arithm}
        Let $n>3$. A lattice $\Gamma \subset \PO(n,1)$ will be called \emph{pseudo-arithmetic} if its ambient
        group is pseudo-admissible.
\end{definition}

It follows from the definition that (quasi-)arithmetic lattices are pseudo-arithmetic. Theorem
\ref{thm:k-G-glueings} shows that lattices obtained by glueing arithmetic pieces are
pseudo-arithmetic. 

In Sect.~\ref{sec:appendix} we will present a method to test if a given
reflection group is pseudo-arithmetic or not; this can be thought as an extension of Vinberg's 
(quasi-)arithmeticity criterion \cite{Vinb67}. We have applied our method on the full list of
groups provided with the software CoxIter (see \cite{coxiter}), which contains about a hundred
non-quasi-arithmetic Coxeter groups for $n>3$: all of them turn out to be pseudo-arithmetic. A particularly interesting example 
is the Coxeter group $\Delta_5$ presented in Figure~\ref{fig:Cox-5}; it has recently been shown by
Fisher et.\ al.\ \cite[Sect.~6.2]{Fisher-et-al} that $\Delta_5$ is \emph{not} commensurable with any lattice obtained by glueing arithmetic pieces.
At this point it is natural to ask:
\pgfmathsetmacro{\minsize}{0.2cm}
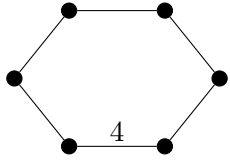
\begin{figure}[h]
\begin{tikzpicture}[scale=0.9]
\tikzstyle{every node}=[draw,shape=circle,minimum size=\minsize,inner sep=0, fill=black]
\node at (0,1) (1) {};
\node at (0.8,0) (2) {};
\node at (2.2,0) (3) {};
\node at (3,1) (4) {};
\node at (2.2,2) (5) {};
\node at (0.8,2) (6) {};
\draw (1)--(2) ;
\draw (2)--(3) node[midway,above=1pt, opacity=0, text opacity=1]  {4};
\draw (3)--(4) ;
\draw (4)--(5) ;
\draw (5)--(6) ;
\draw (6)--(1) ;
\end{tikzpicture}
        \caption{A pseudo-arithmetic Coxeter $5$-simplex group: $\Delta_5$}
        \label{fig:Cox-5}
\end{figure}

\begin{question}
        \label{quest:all-pseudo}
        Let $\Gamma \subset \PO(n,1)$ be a lattice, with $n>3$. Is $\Gamma$ necessarily
        pseudo-arithmetic?
\end{question}

\begin{rmk}
        In dimension $n = 3$ the abundance of hyperbolic manifolds obtained by Dehn surgeries makes
        very unlikely a positive answer to an analogue of this question.
\end{rmk}

Note that our definition implies that the trace field $K$ of a pseudo-arithmetic lattice is (real)
multiquadratic over a totally real number field, and thus $K$ itself is totally real. The following
question -- thus weaker than Question \ref{quest:all-pseudo} -- seems easier to answer, yet we do
not know its current status.

\begin{question}
        \label{quest:all-tot-real}
        Let $n>3$. Is the trace field of a lattice $\Gamma \subset \PO(n,1)$ necessarily totally real? 
\end{question}

\subsection{First type lattices}
\label{sec:first-type}

Let $k \subset \R$ be a number field.
A pseudo-arithmetic lattice over $K/k$ is said to be \emph{of the first type} if its ambient group is 
of the form $\G = \POO_{f, K}$ (resp., $\POp_{f,K}$), where $\POO_f$ denotes the projective orthogonal
group of an admissible quadratic form $f$ over $k$ (resp., its identity component); see
Sect.~\ref{sec:models} for details.
This extends the terminology sometimes used in the arithmetic setting. For $n$ even, any $k$-form 
of $\POO$ is of the type $\POO_f$, and thus for those dimensions all pseudo-arithmetic lattices are
trivially of the first type. It follows from part (2) of Theorem~\ref{thm:k-G-glueings} that lattices obtained by glueing 
arithmetic pieces are of the first type, since such pieces contain totally geodesic hypersurfaces
(see Prop.~\ref{prop:hypersurface-k-rat}).
For $n>3$ odd, it is a priori not clear if any pseudo-arithmetic Coxeter group is of the first type
(this is true for $K=k$ though, again by Prop.~\ref{prop:hypersurface-k-rat}), but this turns out to be
true for the full list of reflection groups studied in Sect.~\ref{sec:appendix}.
To summarize, all \emph{currently known} lattices in $\PO(n,1)$ with $n>3$ are
either:
\begin{itemize}
        \item arithmetic, or
        \item pseudo-arithmetic of the first type. 
\end{itemize}
A lattice simultaneously belongs to both categories if and only if it is arithmetic of the first type
(sometimes also called ``standard'' arithmetic lattices).

\subsection{Volumes}
\label{sec:intro-volumes}

Let $\G = \POO_{f,K}$ (or $\POp_{f,K}$) be a pseudo-admis\-sible group over $K/k$ of the first type.
Then we may associate with $\G$ a set of arithmetic lattices in the following  way. Let us 
assume that $K$ is explicitly given by $K = k(\sqrt{a_1}, \dots, \sqrt{a_r})$, i.e., the degree of
$K/k$ is $s = 2^r$. We will assume (as we may) that the quadratic form $f$ is 
diagonal in the variables $x_0, \dots, x_n$, with negative coefficient in $x_0$.
For any multi-index $i \in \left\{ 0, 1 \right\}^r$ we set $\alpha_i = \sqrt{a_1^{i_1}
\cdots a_r^{i_r}}$, and $f_i = f(x_0, \dots, x_{n-1}, \alpha_i x_n)$. All $f_i$ are admissible
quadratic forms over $k$. For each $i \in \left\{ 0,1 \right\}^r$ we choose an arithmetic subgroup
$\Gamma_i \subset \POO_{f_i}(k)$, i.e., a subgroup commensurable with $\POO_{f_i}(\O_k)$.
We say that the set of arithmetic lattices $\left\{ \Gamma_i \;|\; 0\le i\le 2^r \right\}$ is
\emph{subordinated} to $\G$.

We will show that -- in some precise sense -- the homology over $\Q$ of the group $\G(K)$  is generated by
classes associated with the $\Gamma_i$; see Theorem \ref{th:La-generated}. A direct consequence for the
volume is the following.

\begin{theorem}
        Let $\Gamma \subset \PO(n,1)$ be a pseudo-arithmetic lattice of the first type (with $n>3$), and
        $\left\{ \Gamma_i \right\}$ a set of arithmetic lattices subordinated to the ambient group
        of $\Gamma$. Then
        \begin{align}
                \vol(\Gamma\bs\Hy^n) &= \sum_i \beta_i \vol(\Gamma_i\bs\Hy^n) \label{eq:vol-Q}
        \end{align}
        for some $\beta_i \in \Q$.
        \label{thm:vol-odd}
\end{theorem}

For $n$ even the result readily follows from the Gauss-Bonnet theorem, and Theorem~\ref{thm:vol-odd}
has no interest there (but Theorem~\ref{th:La-generated} presumably has). For manifolds obtained by glueing
the result was already known; see \cite[Sect.~1.4]{Eme-quasi}.  Note also that for quasi-arithmetic
lattices (which corresponds to the case $r=0$ above) the result is proved in \cite[Theorem~1.3]{Eme-quasi}, and the condition ``of the first type'' is superfluous.

The covolume of arithmetic lattices is well-understood: it is essentially expressible by means of special values of $L$-functions (see
\cite{Ono66,Pra89}). In particular the values of the summands $\vol(\Gamma_i\bs\Hy^n)$ can be determined up to
rationals. The following example illustrates Theorem \ref{thm:vol-odd} with the case of the Coxeter
group $\Delta_5$ (which was introduced in Sect.~\ref{sec:intro-pseudo-arithm}).

\begin{example}
        \label{ex:Delta_5}
        It was already proved in \cite[Sect.~III]{Vinb67} that the nonuniform Coxeter group $\Delta_5 \subset \PO(5,1)$ is
        nonarithmetic, of trace field $K= \Q(\sqrt{2})$. It is the unique hyperbolic {\em simplex} group that
        is nonarithmetic for $n>3$. Applying the procedure described in Sect.~\ref{sec:Coxeter trace field},
        one finds from its Gram matrix  that $\Delta_5 \subset \POO_f(K)$, with $f =
        f_0$ given in \eqref{eq:f0}. This shows that $\POO_{f,K}$ is pseudo-admissible, and thus $\Delta_5$ is
        pseudo-arithmetic of the first type. A subordinated set of arithmetic lattices 
        is given by $\left\{\Gamma_i = \aPO_{f_i}(\Z) \;|\; i=0,1 \right\}$, where
        \begin{align}
                f_0 &=  -x_0^2 + x_1^2 +\cdots + x_5^2; \label{eq:f0}\\
                f_1 &= -x_0^2 + x_1^2 +\cdots + 2 x_5^2. \nonumber
        \end{align}
        We have (see for instance \cite[Prop.~2.1]{Eme-quasi} and the references therein): 
        \begin{align*}
               \vol(\Gamma_0\bs\Hy^n) &\in \zeta(3) \cdot \Q^\times;\\ 
               \vol(\Gamma_1\bs\Hy^n) &\in \sqrt{2} L(\chi_8, 3) \cdot \Q^\times,
        \end{align*}
        where $L(\chi_8, \cdot)$ is the Dirichlet $L$-function with nontrivial primitive character mod $8$.
        Equivalently, $L(\chi_8, \cdot) = \zeta_{\Q(\sqrt{2})}/\zeta$; see \cite[Sect.~11]{Zagier81}. 

        The covolume of $\Delta_5$ has been studied by Kellerhals, who obtained a closed
        formula\footnote{Note that in \cite{Johnson-al-99} the value on p.350 differs from the one on p.338,
        and the former is not correct.}
        in \cite[Sect.~3]{Johnson-al-99} (where $\Delta_5 = \widehat{AU}_5$).
        Based on this, the following approximation
        has been found in collaboration with Steve Tschantz. Its accuracy has been checked up to 160 digits.  
        \begin{align}
                \label{eq:vol-D5}
                \vol(\Delta_5 \bs \Hy^5) &\approx \frac{73}{2^9\cdot3^2\cdot5} \zeta(3) +
                \frac{1}{2^3\cdot3^2\cdot5} \sqrt{2} L(\chi_8,3)\\
                &= 0.00757347442200786763497722\dots \nonumber
                \label{eq:Delta5-approx}
        \end{align}
\end{example}
Proving a sharp equality in \eqref{eq:vol-D5} seems out of reach with the methods presented in this
article. In particular, the method used to proved Theorem~\ref{thm:vol-odd} does not provide any
information about the $\beta_i$ (or even about their signs).

\subsection{Final remarks} We close this introduction with some natural questions arising in
connection with the results presented above:
\subsubsection{}%
         The definition of pseudo-arithmeticity can be transposed verbatim for lattices in
         $\PU(n,1)$. In particular, it would be interesting to know if all the nonarithmetic
         examples in $\PU(2,1)$ and $\PU(3,1)$ are ``pseudo-arithmetic'' in this sense.  

\subsubsection{}
In \cite[Question~0.4]{GroPS87} Gromov and Piateski-Shapiro famously asked
whether for any lattice $\Gamma \subset \PO(n,1)$ of sufficiently large
dimension the quotient $\Gamma\bs\Hy^n$ admits a nice partition into 
         ``subarithmetic pieces''. In some weak sense (i.e., at the level of homology),
         Theorem~\ref{th:La-generated} gives a positive answer for the class of pseudo-arithmetic
         lattices of the first type. It is not clear however if a statement closer to their original
         formulation can be achieved for those lattices; in particular, if a positive answer to the following could hold:   
         \begin{question} \label{question:GPS}
                 Let $\Gamma \subset \PO(n,1)$ be a pseudo-arithmetic lattice of the first type
                 ($n>3$), and let $\left\{\Gamma_i\right\}$ be a set of arithmetic lattices
                 subordinated to its ambient group. 
                 Does the collection of subsets $\left\{\Gamma \cap \Gamma_i\right\}$ generate the
                 group $\Gamma$? 
         \end{question}
               
         \subsubsection{} Let $\Gamma \subset \PO(n,1)^\circ$ be a lattice, with ambient group $\G$
         over $K$. For each infinite place $v$ we obtain an embedding $\Gamma \subset \G(K_v)$ with
         Zariski-dense image. If $\G$ is pseudo-admissible over $K/k$ then the group
         $\G(K\otimes_\Q \R) = \prod_{v|\infty}\G(K_v)$ contains exactly $[K:k]$ noncompact factors,
         all of which are locally isomorphic to $\PO(n,1)$. That is, all the factors in this product are
         of real rank at most $1$. Note that the latter property trivially holds for the ambient
         group of {\em any} lattice in $\PGL_2(\C)$. This motivates asking the following.
         \begin{question} \label{question:rigidity}
         Let $G$ be a simple real Lie group of rank $1$, and let $\Gamma \subset G$ be a lattice. Assume
         that there exists a representation $\phi: \Gamma \to H$ with Zariski-dense image into a
         simple real Lie group $H$. Is $H$ necessarily of rank $\le 1$?
         \end{question}

\subsection*{Acknowledgment}
We would like to thank Steve Tschantz for his help with the computation in Example
\ref{ex:Delta_5}.


%% file: glueings.tex
\section{Pseudo-arithmeticity of glueings}
\label{sec:glueings}

The goal of this section is to prove Theorem \ref{thm:k-G-glueings}.

\subsection{Rational hyperplanes}
\label{sec:models}

Let $(V, f)$ be an admissible quadratic space over a number field $k$, i.e., $f: V \to k$ is an admissible
quadratic form.  We denote by $\OO(V,f)$, or simply $\OO_f$, the algebraic group of orthogonal
transformations of $(V, f)$. We denote by  $\POO(V,f) = \POO_f$ its adjoint form (it corresponds
to the group $\mathbf{PGO}(V,f)$ in the notation of \cite{BookInvol}). For a field extension $K/k$,
the group of $K$-points $\POO_f(K)$ can be concretely described as the quotient $\GO(V_K, f)/K^\times$,
where $\GO(V_K, f)$ is the group of similitudes of $V_K = V \otimes_k K$ (see
\cite[Sect.~12.A and Sect.~23.B]{BookInvol}). In particular, $\POO(V, f)$ acts $k$-rationally on the
projective space $\PV_\C$, and so does its identity component $\POp(V,f) = \POp_f$. 

The following is a model for the hyperbolic $n$-space $\Hy^n$, with group of isometries $\POO_f(\R)$:
\[
  \Hy(V, f) = \left\{ x \in \P V_\R \;|\; f(x) < 0 \right\}.
\]
Any hyperplane in $\Hy(V, f)$ corresponds to the image of a subspace $W \subset V_\R$, with $W =
e^\perp$ for some $e \in V_\R$ such that $f(e) > 0$. Such a hyperplane will be denoted by $\Hy(W) \subset
\Hy(V, f)$.  We shall say that $\Hy(W)$ is $k$-rational if $W$ is, i.e., $W = e^\perp$ for 
some $e \in V$. This is equivalent to saying that the projective space $\mathbf{P}W$ is
$k$-closed in $\PV_\R$.  A useful characterization is the following.

\begin{lemma} \label{lemma:rho-k-rational}
        Let $\Hy(W) \subset \Hy(V, f)$ be a hyperplane, and let $\rho \in \POO_f(\R)$ denote the
        reflection through $\Hy(W)$. Then $\Hy(W)$ is $k$-rational exactly when $\rho \in \POO_f(k)$.  
\end{lemma}
\begin{proof}
        Let $B$ be the bilinear form associated with $f$.
        If $W = e^\perp$ for some $e \in V = V_k$, it follows immediately from the 
        formula $\rho(x)= x - \frac{B(x, e)}{B(e, e)} x$  that $\rho$ is $k$-rational.  For the
        converse implication, note that  $\mathbf{P}W$ corresponds the set of fixed points of $\rho$, and thus $\rho \in \POO_f(k)$
        implies that $\mathbf{P}W$ (and thus $\Hy(W)$) is $k$-rational since the action is defined
        over $k$.
\end{proof}

\subsection{Sharp hypersurfaces}%
\label{sec:sharp_hypersurfaces}

Glueings are realized along totally geodesic embedded hypersurfaces of finite volume. To simplify the
discussion, in what follows we call \emph{sharp} such a hypersurface embedded in a hyperbolic
manifold $M$ of finite volume. Note that if $M$ is compact then any totally geodesic hypersurface
embedded in $M$ is sharp. 

\begin{prop} \label{prop:hypersurface-k-rat}
        For $n > 3$,  let $\Gamma \subset \PO(n,1)^\circ$ be an arithmetic lattice over $k$ 
        such that $\Gamma\bs \Hy^n$ contains a sharp hypersurface $N$. Then
        \begin{enumerate}
                \item $\Gamma$ is of the first type, i.e., its ambient group is $\POO^+(V,f)$ for some
                        admissible quadratic space $(V, f)$ over $k$.
                \item Any hyperplane lift of $N$ in $\Hy(V, f)$ is $k$-rational. 
        \end{enumerate}
\end{prop}
\begin{proof}
        The first statement is well known; see for instance \cite[Sect.~10]{Meyer17}.
        For (2) we consider $N \subset \Gamma\bs\Hy(V,f)$, with $\Gamma
        \subset \POO_f(k)$. Let $\Hy(W) \subset \Hy(V,f)$ be a hyperplane lift of $N$, and
        let $\rho \in \POO_f(\R)$ be the associated reflection.  Let $\Lambda$ be the stabilizer of
        $\Hy(W)$ in $\Gamma$, so that $N = \Lambda\bs\Hy(W)$. The latter being sharp, we have that
        $\Lambda$ is Zariski-dense in the group of isometries of $\Hy(W)$. It follows that the
        centralizer $\ZZ$ of $\Lambda$ in $\POO_f$ consists in two elements: $1$ and $\rho$ (since 
        this holds over $\C$ for the centralizer of $\PO_n(\C)$ in $\PO_{n+1}(\C)$). But in a $k$-group of order 2, the nontrivial
        element must be a $k$-point (for it is easily seen to be invariant under the
        Galois group). That is, $\rho \in \POO_f(k)$, and the result follows from
        Lemma~\ref{lemma:rho-k-rational}. 
\end{proof}

\begin{cor} \label{cor:N-over-k}
        Let $M$ be an arithmetic hyperbolic $n$-manifold defined over $k$, with $n>3$. Then  
        any sharp hypersurface in $M$ is arithmetic of the first type, and defined over $k$.
\end{cor}
\begin{proof}
        Let $M = \Gamma\bs\Hy(V, f)$, and $N \subset M$ be a sharp hypersurface, i.e., $N =
        \Lambda\bs\Hy(W)$ for $\Lambda$ the stabilizer in $\Gamma$ of some $k$-rational hyperplane $\Hy(W)$.
        The lattice $\Lambda$ is known to be arithmetic; see \cite[Theorem~3.2]{Meyer17}.
        Let $f_0$ be the restriction of $f$ to $W$. Then $\POp(W,f_0)$ is an admissible $k$-group (for
        $\PO(n-1, 1)^\circ$), whose $k$-points contain the arithmetic subgroup $\Lambda$.
        We conclude that $\POp(W, f_0)$ is the ambient group of $\Lambda$.
\end{proof}

\subsection{Extending similitudes}
\label{sec:extend-similitudes}

For a similitude $\phi: (V_0, f_0) \to (V_1, f_1)$ between quadratic spaces, we denote by
$\conj(\phi): \POO(V_0,f_0) \to \POO(V_1,f_1)$ the map $g \mapsto \phi \circ g \circ
\phi^{-1}$.
\begin{lemma} \label{lem:extend-similitudes}
  Let $(V_0,f_0)$ and $(V_1,f_1)$ be two nondegenerate quadratic spaces over a field $k$, each of
  them containing a quadratic subspace $W_i \subset V_i$ ($i=0,1$) of codimension $1$. Assume that
  there exists a similitude $\varphi: W_0 \to W_1$. Then there exists an extension $F/k$ which is at most
  quadratic such that $\varphi$ extends to a similitude $\phi: V_0 \otimes_k F \to \V_1 \otimes_k F$. In
  particular, $\conj(\phi): \POO(V_0,f_0) \to \POO(V_1,f_1)$ is defined over $F$.
\end{lemma}
\begin{proof}
If $V = (V,f)$ is a quadratic space, we write $\lambda V$ for the quadratic space $(V, \lambda f)$.
For $i=0,1,$ we have a decomposition $V_i = W_i \perp S_i$ where $S_i = W_i^\perp$ is of dimension 1.
Thus $S_0$ and $S_1$ are similar; let $\alpha \in k$ be such that $S_0$ is isomorphic to $\alpha S_1$.
Further let $\lambda \in k$ be such that $\varphi: W_0 \to \lambda W_1$ is an isomorphism.
Then if $F$ denotes the field extension $F = k(\sqrt{\lambda/\alpha})$, we have 
$S_0 \otimes_k F \cong \lambda S_1 \otimes_k F$. This allows to extend $\varphi$ to 
an isomorphism  $V_0 \otimes_k F \cong \lambda V_1 \otimes_k F$, as desired.
\end{proof}

\subsection{Hyperbolic pieces}
\label{sec:hyper-pieces}

By a {\em hyperbolic piece} (of dimension $n$) we mean a complete orientable hyperbolic $n$-manifold of finite
volume with a (possibly empty) boundary consisting in finitely many sharp hypersurfaces.
We say that a piece is {\em singular} if it has nonempty boundary, and {\em regular} otherwise (in which case it
is just a hyperbolic manifold in the sense of Sect.~\ref{sec:hyp-isom}).  A singular piece $M$ can
always be embedded in a regular one of same dimension: it suffices to consider the ``double"
$\widehat{M}$ of $M$ obtained by glueing together two copies of $M$ along each boundary
component ($\widehat{M}$ is complete according to \cite[2.10.B]{GroPS87}).

Let $M$ be a singular hyperbolic piece. 
Its universal cover $\tM$ is isometric to an infinite intersection of half-spaces in $\Hy^n$; see
\cite[Sect.~3.5.1]{martelli16}.
The fundamental group $\pi_1 M$ of $M$ identifies with a discrete subgroup $\Gamma \subset \POp(\R)$ stabilizing $\tM$ and such that $M \cong \Gamma \bs \tM$.
It is clear that any two choices of universal covers for $M$ in $\Hy^n$ are conjugate by an isometry,
and thus, up to conjugacy, the discrete subgroup $\Gamma$ is uniquely determined by $M$. The {\em trace field} of $M$ is
defined as the trace field of $\Gamma$ (see Sect.~\ref{sec:trace-ambient}). Furthermore, by \cite[1.7.B]{GroPS87} we have that $\Gamma$
is Zariski-dense in $\POp(\R)$, and using Theorem \ref{thm:Vinberg} there is 
therefore an intrinsic notion of ambient group for $\Gamma$.

\subsection{Arithmetic pieces}
\label{sec:arithm-pieces}
We define an \emph{arithmetic piece} $M$ (of dimension $n$) as a singular hyperbolic piece of dimension $n$ that embeds into an arithmetic
hyperbolic manifold $M_0= \Gamma_0\bs\Hy^n$.
Thus $M = \Gamma\bs\widetilde{M}$, where $\widetilde{M} \subset \Hy^n$ and
$\Gamma = \Stab_{\Gamma_0}(\widetilde{M})$ is a subgroup of infinite index in $\Gamma_0$ (see \cite[2.10.A]{GroPS87}).

\begin{lemma} \label{lemma:same-trace-field}
        Let $M \subset M_0$ be an arithmetic piece as above, of dimension $n > 3$. Then $M$ and $M_0$ have the same trace
        field and the same ambient group.
\end{lemma}
\begin{proof} 
        Let $k$ be the field of definition of $M_0$, and $K$ the trace field of $M$.
        Clearly we have $K \subset k$.  Now $M$ being singular, it contains a sharp hypersurface $N$ in
        its boundary, whose field of definition must be $k$ by  Corollary~\ref{cor:N-over-k}.
        Using Remark \ref{rmk:field-of-def} it is easily checked that the field of definition of $N$
        (seen as a $(n-1)$-manifold) corresponds to the trace field of $\pi_1(N)$, seen as a subgroup of $\PO(n,1)$.
        We conclude that $k \subset K$, and thus $M$ and $M_0$ (resp.\ their fundamental groups
        $\Gamma$ and $\Gamma_0$) have the same trace field. Let $\G$ be the ambient group of
        $\Gamma_0$. From the inclusion $\Gamma \subset \Gamma_0$ we have $\Gamma \subset \G(k)$, and
        it follows that $\G$ is the ambient group of $\Gamma$ as well (by uniqueness).
\end{proof}

We will say that an arithmetic piece $M$ (resp.\ its fundamental
group) is {\em defined over $k$} if its trace field is $k$; this extends the terminology used for regular pieces (see
Remark \ref{rmk:field-of-def}). 

\begin{cor}
        \label{cor:glue-same-k}
        For $i = 0, 1$ let $M_i$ be an arithmetic piece containing a hypersurface $N_i$ in its
        boundary. If $N_0$ is isometric to $N_1$ then $M_0$ and $M_1$ are defined over the same
        field. 
\end{cor}
\begin{proof}
        Being isometric we have that $N_0$ and $N_1$ have the same field of definition, and the result 
        follows immediately from Corollary \ref{cor:N-over-k} and Lemma~\ref{lemma:same-trace-field}.
\end{proof}

\subsection{Glueings}
\label{sec:glueings-proofs}

Let $M$ be a hyperbolic manifold obtained by glueing arithmetic pieces.  
By definition this means that there exists a sequence 
$$M_1, \dots, M_r = M,$$
where $M_1$ is an arithmetic piece, and for $i \ge 1$ the piece $M_{i+1}$ is obtained from $M_i$ by one 
of the two following operations:
\begin{enumerate}[(I)]
\item glueing an arithmetic piece $M_0$ to $M_i$ using an isometry $\phi$ between two
  hypersurfaces in their respective boundary, i.e.,  
  \[
    M_{i+1} = M_0 \cup_{\phi} M_i.
  \]
\item ``closing up'' two hypersurfaces in the boundary of $M_i$, i.e., $M_{i+1} = M_i / \left\{x = \phi(x)\right\}
        $ where $\phi$ is an isometry between two hypersurfaces in the boundary of $M_i$.
\end{enumerate}
Theorem \ref{thm:k-G-glueings} is a specialization the following.

\begin{theorem}
        \label{thm:trace-group-of-pieces}
        Let $M$ be a hyperbolic piece of dimension $n>3$ obtained by glueing a finite number 
        of arithmetic pieces $M_i$ defined over $k$.
  \begin{enumerate}
  \item The trace field of $M$ is a multiquadratic extension $K/k$.
  \item The ambient group of $M$ is $\G_K$, where $\G$ is the ambient group of any of the piece
          $M_i$.
  \end{enumerate}
\end{theorem}

\begin{proof}
        We proceed inductively. If $M$ is an arithmetic piece over $k$ then (1) and (2) hold trivially.
        Suppose then that $M_1$ is a piece constructed by glueing arithmetic pieces over $k$,
        and  that it respects the conditions (1) and (2) (recurrence assumption).
        We will glue $M_1$ using a sharp hypersurface $N_1 \subset M_1$. Let $K_1$ be the trace field of $M_1$,
        which is multiquadratic over $k$. Let $\G = \POp(V, f)$ be the ambient group of an
        arithmetic piece containing $N_1$, so that $\G_{K_1}$ is the ambient group of $M_1$. Thus
        we may write $M_1 = \Gamma_1\bs \widetilde{M_1}$, for $\widetilde{M_1} \subset \Hy(V, f)$ and
        $\Gamma_1$ a discrete subgroup of $\G(K_1)$ stabilizing $\widetilde{M}$.
       
        Let us first examine the case of the glueing of type (I), i.e., $M$ is obtained by glueing $M_1$ to an
        arithmetic piece $M_0$ along a hypersurface $N_0 \subset M_0$ isometric to $N_1$.
        Abstractly we have $\pi_1(M) = \left<\pi_1(M_0), \pi_1(M_1)\right>$.  
        By Corollary~\ref{cor:glue-same-k} we have that $M_0$ is defined over $k$. Moreover, it is
        of the first type (Prop.~\ref{prop:hypersurface-k-rat}), and we can write
        $M_0 = \Gamma_0\bs \widetilde{M_0}$, where $\widetilde{M_0} \subset \Hy(V_0, f_0)$
        for some admissible quadratic space $(V_0, f_0)$ over $k$, and $\Gamma_0 \subset
        \POp_{f_0}(k)$.
        We fix a hyperplane lift $\Hy(W_0)$  of $N_0$  in $\Hy(V_0, f_0)$, which is necessarily
        $k$-rational by Prop.~\ref{prop:hypersurface-k-rat}.
        We fix as well a hyperplane lift $\Hy(W_1) \subset \Hy(V, f)$ of $N_1$; it is $k$-rational
        for the same reason.  Since $N_1$ and $N_0$ are isometric, the quadratic spaces $W_1$ and $W_0$ must be similar
        over $k$ (see~\cite[2.6]{GroPS87}). More precisely, there exists a similitude $\varphi: W_0 \to W_1$
        whose extension to $W_0 \otimes_k \R$ lifts the glueing isometry between $N_0$ and $N_1$.
        By Lemma~\ref{lem:extend-similitudes} there exists an extension $F/k$, at most quadratic, over which
        $\varphi$ extends to a similitude $\phi: V_0 \otimes_k F \to V \otimes_k F$ (and note that
        the scaling factor can be chosen positive here, so that $F$ is real). In particular,
        $\conj(\phi)(\Gamma_0) \subset \G(F)$. The map $\phi$ induces
        an isometry between $\Hy(V_0,f_0)$ and $\Hy(V, f)$ that matches the lift hyperplanes of
        $N_0$ and $N_1$. Possibly after composing with the reflection through $\Hy(W_1)$ (which is
        $k$-rational by Lemma~\ref{lemma:rho-k-rational}), we can hence see the respective universal
        covers $\widetilde{M_0}$ and $\widetilde{M_1}$ of $M_0$ and $M_1$ as lying on both sides of $\Hy(W)$ in
        $\Hy(V, f)$. We can thus
        write $M = \Gamma\bs\Hy(V, f)$, where $\Gamma = \left< \conj(\phi)(\Gamma_0),
        \Gamma_1\right>$.  It follows $\Gamma \subset \G(F_1)$, where $F_1$ is the composite of $F$
        and $K_1$. In particular, the trace field $K$ of $\Gamma$ is either $K_1$ or $F_1$, which
        in either case is multiquadratic over $k$. It follows that $\G_K$ is the ambient group
        of $\Gamma$.
        
        Finally note that the inductive proof can be started with any arithmetic piece. Thus the group $\G$
        can be the ambient group of any of these pieces, as stated in (2). 

        For a glueing of type (II), we consider two isometric sharp hypersurfaces $N_0, N_1 \subset M_1$.
        If $M_0 \subset M_1$ is an arithmetic piece containing $N_0$ and $(V_0, f_0)$ and $W_0$ are defined as above,
        the same argument gives a similitude $\phi: V_0 \otimes_k F \to V \otimes_k F$, where $F/k$ is an extension
        which is at most quadratic.
        Now by the induction hypothesis, $\POp(V_0, f_0)_{K_1}$ is an ambient group for
        $\Gamma_1$.
        It follows that we have an isometry $\psi: \Hy(V_0, f_0) \to \Hy(V,f)$ such that
        $\conj(\psi)$ is defined over $K_1$ and 
        $g = \phi \circ \psi^{-1}$ (seen as an element of $\G(\R)$) lifts an isometry $N_1 \cong N_0$.
        Now by construction, there exists $g' \in \G(K)$ such that $\Gamma = \left< \Gamma_1 , g' \right>$, and
        $g'$ induces the isometry $N_1 \cong N_0$.
        It follows that $g$ and $g'$ coincide on $\Hy(W_1)$, and thus $g \in \{g', g' \rho\}$ where $\rho$ denotes
        reflection through $\Hy(W_1)$.
        As before, this shows that $\Gamma \subset \G(F_1)$, where $F_1$ is the composite of $F$ and $K_1$.
        Thus the trace field $K$ of $\Gamma$ is multiquadratic over $k$, and $\G_K$ is the ambient group
        of $\Gamma$.
\end{proof}


%% file: homology.tex
\section{Homology and Volumes}
\label{sec:homology}

The main result of this section is Theorem~\ref{th:La-generated}, of which
Theorem~\ref{thm:vol-odd} will be a direct consequence. Since these results hold up to 
commensurability, we will restrict ourselves to considering lattices that 
are torsion-free and orientation-preserving. The basic idea is, given an ambient group $\G_K$, to
compare fundamental classes of lattices $\Gamma \subset \G(K)$ in a $\Q$-vector space $\La(\G)$; the
latter will be defined in Sect.~\ref{sec:fund-class-volum}.
Many of the proofs are generalizations 
of the arguments in \cite{Eme-quasi}, which deals with quasi-arithmetic lattices.
Prop.~\ref{prop:homology-generated} is completely new; it is the key result needed at the level of
ambient groups. Its proof is given in Sect.~\ref{sec:homology-surj} 

\subsection{Fundamental classes and volumes} \label{sec:fund-class-volum}
Let $G = \PO(n,1)$ and let $\Omega = \partial \Hy^n$ denote the geometric boundary of $\Hy^n$.
Any subgroup $S \subset G$ acts naturally on $\Omega$; we write
$H_n(S, \Omega)$ for the homology of $S$ \emph{relative to $\Omega$}; see \cite[Sect.~1.6]{Eme-quasi}.
By definition $H_n(S, \Omega)$ is the homology group $H_{n-1}(S; J\Omega)$ with coefficients in the
kernel of the augmentation map $\Z\Omega \to \Z$, and it follows that there is a canonical
``connecting'' map  from the usual homology:
\begin{align}
        \label{eq:delta}
        \delta:& H_n(S) \to H_n(S, \Omega).  
\end{align}
If $S$ acts freely on $\Omega$, this map is an isomorphism.

In the case where $S = \Gamma\subset G^\circ$ is a torsion-free lattice, we have that $H_n(\Gamma, \Omega)$ 
is isomorphic to the singular homology of the end-compactification $\widehat M$ of
$M = \Gamma \bs \Hy^n$ (see \cite[Sect.~3]{Eme-quasi}).
Thus $H_n(\Gamma, \Omega) \cong \Z$; let $\xi$ denote the generator that corresponds 
to the positive orientation in $H_n(\widehat M)$. We define the \emph{fundamental class} $[\Gamma] \in
H_n(G, \Omega; \Q)$ as the image of $\xi$ under the map induced by the inclusion $\Gamma \subset G$.
This notion captures the volume:
\begin{prop} \label{prop:vol-homology}
  There exists a linear map $v: H_n(G, \Omega; \Q) \to \R$ such that $v([\Gamma]) =
  \vol(\Gamma \bs \Hy^n)$ for any torsion-free lattice $\Gamma \subset G^\circ$.
\end{prop}
\begin{proof}
See \cite[Prop.~1.7]{Eme-quasi}.
\end{proof}

\subsection{The space $\La(\G)$}%
\label{sec:La-G}

Let $K \subset \R$ be a number field, and let $\G$ be a connected $K$-group such that $\G(\R)^\circ
= \PO(n,1)^\circ$; for the moment we are not assuming that $\G$ is admissible (nor
pseudo-admissible).
The inclusion $\G(K) \subset \G(\R)$ composed with the map $\delta$  in \eqref{eq:delta} induces a map 
\begin{align}
        H_n(\G(K); \Q) \to H_n(\G(\R), \Omega; \Q),
\end{align}
whose image will be denoted by $\La(\G)$. If $\G = \G_{0,K}$ for some $k$-group $\G_0$ and some
extension $K/k$ (so that $\G_0(\R) = \G(\R)$), then $\La(\G_0) \subset \La(\G)$ (where $\La(\G_0)$ is the image of $H_n(\G_0(k);
\Q)$).


\begin{lemma} \label{lem:fund-in-La}
        Let $\Lambda \subset \G(K) \cap \G(\R)^\circ$ be a torsion-free lattice.
  Then $[\Lambda]$ lies in $\La(\G)$.
\end{lemma}
\begin{proof}
  This follows from the commutative diagram (6.2) and Prop.~6.1 of \cite{Eme-quasi}, whose proof
  works in our more general setting.
\end{proof}

This last result permits to see the fundamental class of a lattice $\Gamma$ as an element of
$\La(\G)$, where $\G$ is the ambient group of $\Gamma$. By Prop.~\ref{prop:vol-homology} the class
$[\Gamma]$ determines the volume; the space $\La(\G)$ might be thought as a replacement for the
Bloch group, which appears in the case of dimension $n=3$ (see for instance
\cite[Sect.~12.7]{MaclReid03}). In the rest of this section we will show that for $\G$
pseudo-admissible of the first type, one can exhibit generators for $\La(\G)$ corresponding to 
classes of arithmetic lattices.

\begin{rmk} \label{rmk:Bloch}
        The definition of $\La(\G)$ is comparable to the ``higher groups'' $\mathscr{P}_n$ in Neumann
        and Yang \cite[Sect.~8]{NeuYan99}, with the difference (crucial for our applications) that we work with $K$-points instead
        of $\R$-points. 
\end{rmk}


\subsection{Subordinated fundamental classes}%
\label{sec:subordinated_fundamental_classes}

Assume now that $\G$ is connected and  pseudo-admissible of the first type over a multiquadratic extension $K/k$.
That is, there exists an admissible quadratic form $f$ over $k$ such that $\G = \aPO_{f,K}^+$.
As in Sect.~\ref{sec:intro-volumes} we write $f = f(x_0, \dots, x_n)$ in a diagonal form with negative $x_0$ coefficient,
and $K = k(\sqrt{a_1}, \dots, \sqrt{a_r})$ with $a_j \in k$. For $i = (i_1, \dots, i_r) \in
\{0,1\}^r$  we  set $\alpha_i = \sqrt{ {a_1}^{i_1} \cdots {a_r}^{i_r}}$.
Then the $\alpha_i$ form a basis of $K/k$.

We let $\G_i$ denote the $k$-group $\POp_{f_i}$, where  $f_i = f(x_0, \dots, x_{n-1}, \alpha_i x_n)$. 
The diagonal matrix $D_i$ with diagonal entries $1, \dots, 1, \alpha_i$ satisfies $f_i = f \circ D_i$, and
induces (via conjugation) an isomorphism of algebraic groups $\G_{i,K} \xrightarrow{\sim} \G$.
In particular, it induces an inclusion $\G_i(k) \inj \G(K)$.
The proof of the following proposition is the subject of Sect.~\ref{sec:homology-surj}.

\begin{prop} \label{prop:homology-generated}
  The following natural map, induced by the inclusions $\G_i(k) \inj \G(K)$,  is surjective:
  \[
    \Oplus_{i\in\{0,1\}^r} H_n(\G_i(k); \Q) \longto H_n(\G(K); \Q).
  \]
\end{prop}

Assuming the proposition, we are ready to state and prove the main theorem of this section.
Together with Prop.~\ref{prop:vol-homology} it directly implies Theorem~\ref{thm:vol-odd}. 
\begin{theorem} \label{th:La-generated}
  Let $\G$ be connected and pseudo-admissible over $K/k$ of the first type.
  For each $i \in \{0,1\}^r$, let $\Gamma_i$ be a torsion-free arithmetic lattice in ${\G_i(k)\cap G^\circ}$.
  Then the $\Q$-vector space $\La(\G)$ is generated by the set of fundamental classes
  $\left\{[\Gamma_i] \;|\; i \in \{0,1\}^r\right\}$.
\end{theorem}
\begin{proof}
  The algebraic groups $\G_i$ are admissible over $k$, and therefore \cite[Prop.~4.2]{Eme-quasi} shows
  that $H_n(\G(k); \Q)$ has dimension one.
  The fact that the volume of $\Gamma_i \bs \Hy^n$ is non-zero combined with Prop.~\ref{prop:vol-homology} and
  Lemma~\ref{lem:fund-in-La} implies that $\La(\G_i)$ is of dimension one and generated by $[\Gamma_i]$.
  Now Prop.~\ref{prop:homology-generated} shows that $\La(\G)$ is generated by its subspaces $\La(\G_i)$,
  and this finishes the proof.
\end{proof}

\begin{rmk} \label{rmk:dim-LaG}
        We do not claim that the classes $[\Gamma_i]$ are linearly independent. For $n$ odd it might 
        well be the case, since there are no expected congruences for the special values of
        $L$-functions that appears in the covolumes. But this justification certainly fails for $n$
        even, and it is not excluded that $\La(\G)$ has dimension $1$ there.
\end{rmk}


\subsection{The homology of ambient groups}
\label{sec:homology-surj}

The goal of this section is to prove Prop.~\ref{prop:homology-generated}.
We begin by relating the homology of an algebraic group to that of its simply-connected cover.
\begin{lemma} \label{lem:homology-simply-connected}
  Let $\G$ be a connected simple adjoint algebraic group defined over a field $K$, and let $\sG$ denote its
  simply-connected cover.
  The natural map $\sG(K) \to \G(K)$ induces a surjective map
  \[
  H_n(\sG(K); \Q) \to H_n(\G(K);\Q).
  \]
\end{lemma}
\begin{proof}
  We borrow the notation from the proof of \cite[Prop.~4.2]{Eme-quasi}, since our proof is essentially the same.
  In particular, $C$ denotes the center of $\sG$. 
  Consider the following composition of maps
  \[
    \begin{tikzcd}[row sep=small, column sep =small]
      H_n(\sG(K); \Q) \rar & H_n(\sG(K)/C(K); \Q) \dar &\\
      & H_n(\sG(K)/C(K); \Q)_A \rar & H_n(\G(K); \Q) 
    \end{tikzcd}
  \]
  where the horizontal maps are induced by the quotient $\sG(K) \surj \sG(K)/C(K)$ and the inclusion
  $\sG(K)/C(K) \inj \G(K)$ coming from the exact sequence (4.3) in \cite{Eme-quasi}, respectively.
  
  These horizontal maps coincide with the edge homomorphisms of the spectral sequences (4.5) and (4.4) in
  \cite{Eme-quasi} respectively, 
  which are isomorphisms since both sequences collapse at $E^2$ (see \cite[6.8.2]{Weib94}).
  The vertical map is obviously surjective, and looking at the definition of edge homomorphisms it is routine to check
  that the composition of the three maps is induced by $\sG(K) \to \G(K)$.
\end{proof}

\begin{proof}[Proof of Proposition~\ref{prop:homology-generated}]
Let $\lambda_i$ be the inclusion $\G_i(k) \to \G(K)$ defined in
Sect.~\ref{sec:subordinated_fundamental_classes}. Since $\lambda_i$ is defined by conjugating the
quadratic form $f$, we may lift it at the level of the Spin groups, i.e., at the level of the
simply-connected covers of $\G_i$ and $\G$.  Now using Lemma~\ref{lem:homology-simply-connected}, 
we see that to prove Prop.~\ref{prop:homology-generated} it suffices to show that the induced map
\begin{align} \label{eq:lift-lambda}
    \Oplus_{i\in\{0,1\}^r} H_n(\sG_i(k); \Q) \longto H_n(\sG(K); \Q)
\end{align}
is surjective. We will show that it is actually an isomorphism.

Recall (see \cite[Sect.~6.3]{Burgos} and \cite[Sect.~4]{Eme-quasi}) that for the cohomology with coefficients in $\R$ we
have two canonical isomorphisms from the continuous cohomology of $\sG_i(k\otimes_\Q \R)$: 
\begin{align} \label{eq:cohom-cont}
        H^n(\sG_i(k); \R) \xleftarrow{\sim} H^n_\mathrm{ct}(\sG_i(k\otimes_\Q \R);\R) \xrightarrow{\sim} H^n(X_i; \R),
\end{align}
where $X_i$ is the compact dual space (or compact twin) associated with $\sG_i(k\otimes_\Q \R)$.
For $\G_i$ admissible, $X_i$ is the $n$-sphere;
more precisely, $X_i$ is the symmetric space associated with $\hat f_i$, the positive definite version of $f_i$ (i.e.,
$f_i$ with the coefficient of $x_0$ made positive).
It follows that the left hand side of \eqref{eq:lift-lambda}
has dimension $s = 2^r$. The same argument shows that the dimension of $H^n(\sG(K); \R)$ equals that of
$H^n(X; \R)$, where the compact dual $X$ associated with $\sG(K\otimes_\Q \R)$ is a product $S^n
\times \cdots \times S^n$ of $s$ copies of the $n$-sphere; this dimension is $s$ by the Künneth
formula.


To prove that \eqref{eq:lift-lambda} is indeed an isomorphism it clearly suffices to show that the
same map with coefficients extended to $\R$ is an isomorphism. This allows (after dualizing the
cohomology) to use the canonical maps in \eqref{eq:cohom-cont} and their analogues for $\sG$.
Observe that the first map in \eqref{eq:cohom-cont} is induced by the inclusion $k \subset k \otimes_\Q \R$ and
the second is just the natural isomorphism between the corresponding Lie algebras (combined with multiplication
by a scalar, see \cite[Sect.~6.3]{Burgos}).
Both maps are functorial in the sense that they commute with the maps induced by
the inclusions $\G_i(k) \to \G(K)$.
Then it is enough to show that the map  
\begin{align} \label{eq:hom-cpt-dual}
        \theta: \Oplus_{i\in\{0,1\}^r} H_n(X_i; \R) \longto H_n(S^n \times \cdots \times S^n; \R) 
\end{align}
is an isomorphism.
For the multi-index $0 = (0,\dots,0)$ we fix a generator $[X_0]$ of $H_n(X_0; \R)$.
In \eqref{eq:cohom-cont} the only contribution to the cohomology of ${\sG_0(k \otimes_\Q \R)}$ comes
from the noncompact factor $\sG_0(\R)$. Similarly for $\sG(K\otimes_\Q \R)$, whose noncompact
factors form a product $\prod_{\sigma} \sG_0(\R)$ indexed by the set of embeddings $\sigma: K \to \R$
fixing $k$ (since $\sG = \sG_{0,K}$), that is, by the Galois group $\Gal(K/k)$ since $K/k$ is Galois.
The Künneth formula then gives an isomorphism
\[
  H_n(S^n \times \cdots \times S^n; \R) \cong \Oplus_{\sigma} H_n(S^n; \R),
\]
with the right hand side indexed by $\Gal(K/k)$.
Let $([S_\sigma])_\sigma$ denote the image of $[X_0]$ via \eqref{eq:hom-cpt-dual}.

The matrix $D_i = \diag(1,\dots,1,\alpha_i)$ induces a homeomorphism $X_i \cong X_0$ (since
$\hat f_0 \circ D_i = \hat f_i$) and thus an isomorphism in homology; we denote by $[X_i] \in H_n(X_i;\R)$
the image of $[X_0]$.
A careful analysis of the maps $\lambda_i$ then shows that $\theta([X_i]) = ((\sigma(\alpha_i)/\alpha_i) [S_\sigma])_\sigma$.
With respect to the bases $\{ [X_i] \}$ and $\{[S_\sigma]\}$ the map $\theta$
is thus given by the matrix $B = (\sigma(\alpha_i)/\alpha_i)$. Its determinant equals 
\begin{align} \label{eq:det-B}
        \det(B) = \prod_i \alpha_i^{-1} \cdot \Delta[\{\alpha_i\}],
\end{align}
where $\Delta[\{\alpha_i\}]$ is the discriminant in $K/k$ of the basis $\{\alpha_i\}$. Hence $\det(B)$
is nonzero, and the proposition is proved.
\end{proof}


%% file: Coxeter.tex
\section{Reflection groups and pseudo-admissibility}
\label{sec:appendix}

In this section we explain how to check if a given hyperbolic reflection group is pseudo-arithmetic. We
have applied the method to all examples contained in the tables provided  
by the program CoxIter developed by Guglielmetti \cite{coxiter}. We have found that all these
reflection groups are pseudo-arithmetic for $n>3$; see Sect.~\ref{sec:computat} for a summary of the
results.

\subsection{Ambient groups of Coxeter groups}
\label{sec:Coxeter trace field}

The references for this section are Vinberg \cite[Theorem~5]{Vinb71} and Maclachlan \cite[Sect.~9]{Maclachlan}.
Let $q = -x_0^2 + x_1^2 + \cdots + x_n^2$ denote the standard quadratic form in $\Hy^n$, with associated bilinear form
$(x,y) \mapsto B(x,y) = q(x+y) - q(x) - q(y)$.
Let $P \subset \Hy^n$ be a Coxeter polytope with $r$ faces whose orthogonal vectors are denoted by $e_1,\dots,e_r$, and 
are chosen to have $q$-norm one.
Let $\Gamma \subset \PO(n,1)$ be the Coxeter group associated with $P$, and let $G = (a_{ij})_{ij}$ be its Gram matrix; its entries are $a_{ij} = B(e_i,e_j)$.
For a subset $\{ i_1 , \dots, i_k \} \subset \{1, \dots, r\}$, define
\[
b_{i_1, \dots, i_k} = a_{i_1i_2} \cdots a_{i_{k-1} i_k} a_{i_k {i_1}} \quad \text{and} 
\quad v_{i_1, \dots, i_k} = (a_{1 i_1} a_{i_1 i_2} \cdots a_{i_{k-1} i_k}) \, e_{i_k}.
\]
The trace field of $\Gamma$ is then the field $K = \Q(\{b_{i_1, \dots, i_k}\})$ generated by all the possible
``cyclic'' products $b_{i_1, \dots, i_k}$, and the ambient group is the group $\aPO_{f}$ where $f$ is the
quadratic form $q$ restricted to the $K$-vector space $V$ spanned by the $v_{i_1, \dots, i_k}$.
Note that the quadratic form $f$ can be determined
without explicitly computing the $e_i$; one proceeds as follows:
\begin{enumerate}
\item Extract a minor $M$ of the Gram matrix which has full rank $n+1$. 
  This minor is then given by $M = (B(e_{m_i}, e_{m_j}))_{ij}$ for some sequence $m_1, \dots, m_{n+1}$, and 
  it follows that the vectors $e_{m_1}, \dots e_{m_{n+1}}$ are linearly independent over $\R$.
\item For each  $j = 1, \dots, n+1$, find a sequence
  of indices $l_1, \dots, l_k$ with $l_k = m_j$ such that the product
  $a_{1 l_1} a_{l_1 l_2} \cdots a_{l_{k-1} l_k}$ is non-zero (possible since the Coxeter diagram is
  connected), and set $w_j = v_{l_1, \dots, l_k}$.
\item The $K$-vector space $V$ then admits $\{w_1, \dots, w_{n+1}\}$ as a basis.
  It follows that $f$ has the matrix representation $A = (B(w_i, w_j))_{ij}$, whose entries can be easily 
  computed from the entries of $G$.
\end{enumerate}

\subsection{Quadratic forms over multiquadratic extensions} \label{sec:qforms multiquadratic}

For this section we refer to Lam's book \cite{Lam} and the two papers \cite{ELWoriginal,ELWsequence}
of Elman, Lam and Wadsworth. 
If $k$ is a field, $W(k)$ will denote the Witt group (or ring) of $k$ (as defined in \cite[Chap.~II]{Lam}).
Recall that two rank $n$ regular quadratic forms $f,f'$ over $k$ are isometric if and only if they represent the same 
element in $W(k)$, and that $0 \in W(k)$ is represented by any hyperbolic form.
Since we will use this group only to compare regular forms of the same rank, we will usually identify quadratic forms 
with their image in $W(k)$.

For a field extension $K/k$, we let $r_k^K:W(k) \to W(K)$ denote the group homomorphism induced by the inclusion 
$k \subset K$ (sending a quadratic form $f$ to $f_K$, the same form seen over $K$).
If $K/k$ is a quadratic extension with $K = k(\sqrt{a})$, $a \in k$, let $s: K \to k$ be the $k$-linear map
sending $1 \mapsto 0$ and $\sqrt{a} \mapsto 1$.
This map induces a group homomorphism $s_k^K: W(K) \to W(k)$ called the \emph{transfer map}:
a quadratic form $q\in W(K)$ of rank $n$ is mapped to $s_k^K(q) = s \circ q$, seen as a quadratic form of rank $2n$ over $k$.
See \cite[Sect.~VII.3]{Lam} for details.

Let now $K/k$ be a multiquadratic extension of a global field $k$.
The two results we will need for our criterion are:
\begin{enumerate}
\item $K/k$ is an \emph{excellent} extension, i.e., a quadratic form $q$ is defined over $k$
  (by which we mean $q \cong f_K$ for some quadratic form $f$ over $k$)
  if and only if $q \in \im (r_k^K)$.
  This is a consequence of \cite[Theorem~2.13]{ELWoriginal} which states that any extension of a global field $k$ containing a 
  Galois extension of $k$ of even degree is excellent.
  See also \cite[Sect.~XII.4]{Lam}.
\item There is an exact sequence of Witt group
\[
\begin{tikzcd}
  W(k) \rar["r_k^K"] & W(K) \rar["\prod s_F^K"] & \displaystyle \prod W(F),
\end{tikzcd}
\]
where the product on the right is taken on all fields $k \subset F \subset K$ with $[K:F] = 2$.
This is the end of the sequence (1.1) in \cite{ELWsequence}, whose exactness is proven for global fields in Theorem C of the same paper.
\end{enumerate}

From these two results we deduce:
\begin{prop}\label{prop:criterion}
  A regular quadratic form $q\in W(K)$ is defined over $k$ if and only if for all subfields
  $k \subset F \subset K$ with $[K:F] = 2$, 
  we have $s_F^K(q) = 0 \in W(F)$, that is, if the corresponding form is hyperbolic.
\end{prop}

\subsection{The computations}
\label{sec:computat}
Let $\Gamma \subset \PO(n,1)$ be a Coxeter group with trace field $K$ and ambient group $\aPO_f$
for some quadratic form $f$ over $K$.
In order to check if $\Gamma$ is pseudo-admissible, one needs to find a subfield $k \subset K$
such that $K/k$ is multiquadratic and a $k$-form of $\aPO_f$ which is admissible. 
Using Prop.~\ref{prop:criterion}, one can already find a candidate for $k$: the intersection
of all codimension 2 subfields $F \subset K$ such that $s^K_F(f)$ is hyperbolic.
If one is then able to find an admissible quadratic form $g$ over $k$ such that $f \cong g_K$,
one can conclude: $\Gamma$ is pseudo-arithmetic over $K/k$ with ambient group $\aPO_g$.

If however the field $K$ is multiquadratic over $\Q$ and the form $f$ is itself defined over $\Q$,
the computation of such a $g$ becomes superfluous, since any quadratic form over $\Q$ of signature
$(n,1)$ is admissible.
It turns out that all the examples in the table of Coxeter groups of \cite{coxiter} fulfill
this simpler necessary condition for pseudo-arithmeticity.

The computations were done using the mathematical program Sage \cite{sagemath} and an algorithm of
\cite{KC} for checking hyperbolicity over number fields.
In Tables \ref{tab:1}--\ref{tab:2}, we show a sample of the Coxeter groups we analyzed, namely those from the classification of 
Roberts \cite{Roberts} having trace field of degree $>1$. 
Each of these Coxeter groups is pseudo-arithmetic over $K/\Q$ (where $K$ is the trace field).
The indicated quadratic forms were found by brute-force search over quadratic forms of the
shape $f_0 + ax_n^2$, where $f_0 = -x_0^2 + x_1^2 + \cdots + x_{n-1}^2$.

The two examples of Coxeter groups with trace field of degree $8$ over $\Q$ are shown in Figure \ref{fig:deg8}.
Moreover, the 5 examples of Coxeter groups in $\Hy^5$ with trace field of degree 4 over $\Q$
are shown in Figure \ref{fig:deg4}.

\newcolumntype{L}[1]{>{\raggedright\let\newline\\\arraybackslash\hspace{0pt}}p{#1}}
\newcolumntype{M}[1]{>{\raggedright\let\newline\\\arraybackslash\hspace{0pt}}m{#1}}
\begin{table}[h!] 
\begin{tabular}{|c|M{0.21\textwidth}|c|l|M{0.21\textwidth}|}
  \hline
    deg & Reference \newline \cite[page(item)]{Roberts}   & dim & Trace field & Quadratic form \newline $f = f_0 + a x_n^2$\\
  \hline
  \multirow{2}{*}{8} & 9(e) &\multirow{2}{*}{4} & \multirow{2}{*}{$\Q(\sqrt{2}, \sqrt{3}, \sqrt{5})$} 
                                      & \multirow{2}{*}{$a=1$} \\
        & 11(e) &  &&  \\
  
  \cline{1-5}
  \multirow{15}{*}{4}& 9(a) & \multirow{8}{*}{4} & $\Q(\sqrt{2}, \sqrt{15})$ & $a=3$ \\ 
  \cline{4-4}
& 9(c) &  &  \multirow{3}{*}{$\Q(\sqrt{3}, \sqrt{5})$ } & $a=2$ \\
& 9(d) &  &  & $a=1$ \\
& 9(f) &  &  & $a=2$ \\  
  \cline{4-4}\cline{5-5}
& 11(a)  & & $\Q(\sqrt{2}, \sqrt{3})$  & \multirow{4}{*}{$a=1$} \\
  \cline{4-4}
& 11(c)  & &  \multirow{3}{*}{$\Q(\sqrt{5}, \sqrt{6})$} & \\
& 11(d) &  &  & \\
        & 11(f) &  &  &  \\  
  \cline{2-5}
  & 14(c) & \multirow{5}{*}{5} & \multirow{2}{*}{$\Q(\sqrt{2}, \sqrt{3})$} 
                                      &  \multirow{5}{*}{$a = 1$} \\
  & 17(d) & &&  \\
  \cline{4-4} 
  & 14(d)  && \multirow{2}{*}{$\Q(\sqrt{2}, \sqrt{5})$}&  \\
  & 17(c) &  && \\
  \cline{4-4} 
  & 19(e)  && $\Q(\sqrt{2}, \sqrt{13})$& \\
  \cline{2-5}

        & 23(e) & \multirow{2}{*}{6} & \multirow{2}{*}{$\Q(\sqrt{2}, \sqrt{3})$} 
                                      & \multirow{2}{*}{$a = 1$}\\
  & 24(h) &  &  & \\
  \hline
\end{tabular}
\caption{Coxeter groups in \cite{Roberts} with trace field of degree 8 and 4.}\label{tab:1}
\end{table}

\begin{table}[h!] 
\begin{tabular}{|c|L{0.35\textwidth}|c|l|M{0.21\textwidth}|}
\hline
  deg & Reference \cite[page(item)]{Roberts}   & dim & Trace field & Quadratic form \newline $f = f_0 + a x_n^2$\\
\hline
  \multirow{27}{*}{2} &  08(d), 08(e) & \multirow{6}{*}{4}& $\Q(\sqrt{15}) $ & $a = 6$\\
&  08(b), 12(b) & & $\Q(\sqrt{3}) $ & $a = 1$\\
&  10(d), 11(b) & & $\Q(\sqrt{3}) $ & $a = 2$\\
&  08(f) & & $\Q(\sqrt{30}) $ & $a = 3$\\
&  09(b) & & $\Q(\sqrt{30}) $ & $a = 6$\\
&  10(b), 10(c) & & $\Q(\sqrt{6}) $ & $a = 1$\\
\cline{2-5}
&  13(d), 13(f), 16(d), 16(f) & \multirow{10}{*}{5} & $\Q(\sqrt{10}) $ & $a = 2$\\
&  19(d), 19(f) & & $\Q(\sqrt{13}) $ & $a = 1$\\
&  18(b), 18(c), 18(e), 20(b), 20(c), 20(e), 21(b), 21(d), 22(a)& & $\Q(\sqrt{2}) $ & $a = 1$\\
&  19(b), 19(c) & & $\Q(\sqrt{26}) $ & $a = 2$\\
&  14(a), 14(e), 17(b), 17(f) & & $\Q(\sqrt{3}) $ & $a = 1$\\
&  14(b), 14(f), 15(c), 15(e), 17(a), 17(e) & & $\Q(\sqrt{5}) $ & $a = 1$\\
&  13(c), 13(e), 16(c), 16(e) & & $\Q(\sqrt{6}) $ & $a = 2$\\
\cline{2-5}
&  26(d), 26(e) & \multirow{6}{*}{6}& $\Q(\sqrt{10}) $ & $a = 1$\\
&  24(b), 24(c), 25(c), 25(d), 25(f), 25(g) & & $\Q(\sqrt{2}) $ & $a = 1$\\
&  23(d), 23(f), 24(g), 25(a) & & $\Q(\sqrt{3}) $ & $a = 1$\\
&  26(a), 26(b) & & $\Q(\sqrt{35}) $ & $a = 10$\\
&  23(b), 23(c), 24(e), 24(f) & & $\Q(\sqrt{6}) $ & $a = 2$\\
\cline{2-5}
&  27(b), 27(c) & \multirow{4}{*}{7}& $\Q(\sqrt{22}) $ & $a = 2$\\
&  27(e), 27(f) & & $\Q(\sqrt{3}) $ & $a = 1$\\
&  28(e), 28(f) & & $\Q(\sqrt{5}) $ & $a = 3$\\
&  28(b), 28(c), 29(b), 29(c) & & $\Q(\sqrt{6}) $ & $a = 2$\\
\cline{2-5}
&  30(f), 31(a) & \multirow{2}{*}{8}& $\Q(\sqrt{3}) $ & $a = 2$\\
&  30(b), 30(c) & & $\Q(\sqrt{5}) $ & $a = 2$\\
\cline{2-5}
&  32(b), 32(c) & 9& $\Q(\sqrt{5}) $ & $a = 1$\\
&  33(b), 33(c) & 10& $\Q(\sqrt{2}) $ & $a = 1$\\
\hline
\end{tabular}
\caption{Coxeter groups in \cite{Roberts} with trace field of degree 2.}\label{tab:2}
\end{table}

\pgfmathsetmacro{\minsize}{0.2cm}
\renewcommand\thesubfigure{\arabic{subfigure}}
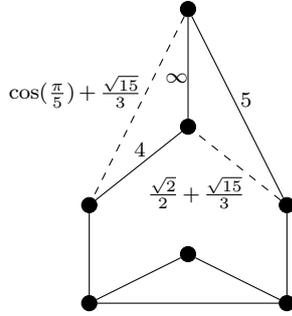
\begin{figure}[h]
\caption{Two examples of Coxeter polytopes in $\Hy^4$ with trace field of degree 8 over $\Q$.} \label{fig:deg8}
\begin{subfigure}[h]{.45\textwidth}
  \begin{centering}
\begin{tikzpicture}[scale=1.3]
\tikzstyle{every node}=[draw,shape=circle,minimum size=\minsize,inner sep=0, fill=black]
\node at (0,0) (61) {};
\node at (0,1) (4) {};
\node at (2,1) (0) {};
\node at (1,3) (3) {};
\node at (2,0) (63) {};
\node at (1,0.5) (62) {};
\node at (1,1.8) (1) {};
\draw (3)--(1)  node[midway,below left, opacity=0, text opacity=1]  {\scriptsize $\infty\,$};
\draw[dashed] (0)--(1) node[midway,below=10pt, left=-4pt, opacity=0, text opacity=1]%
{\scriptsize $\frac{\sqrt{2}}{2}+\frac{\sqrt{15}}{3}$};
\draw[dashed] (3)--(4) node[midway,above=6pt, left=-2pt, opacity=0, text opacity=1]%
{\scriptsize $\cos(\frac{\pi}{5})+\frac{\sqrt{15}}{3}$};
\draw (4)--(61) ;
\draw (0)--(63) ;
\draw (61)--(63);
\draw (61)--(62);
\draw (62)--(63);
%
%
\draw (1)--(4) node[midway,above=6pt, left=-4pt, opacity=0, text opacity=1]  {\scriptsize $4$};
\draw (0)--(3) node[midway,above right, opacity=0, text opacity=1]  {\scriptsize $\,5$};
%
\end{tikzpicture}
\caption{\cite[p. 9 (e)]{Roberts}}
  \end{centering}
\end{subfigure}
\begin{subfigure}[h]{.45\textwidth}\begin{centering}
\begin{tikzpicture}[scale=1.3]
\tikzstyle{every node}=[draw,shape=circle,minimum size=\minsize,inner sep=0, fill=black]
\node at (0,0) (61) {};
\node at (0,1) (4) {};
\node at (2,1) (0) {};
\node at (1,3) (3) {};
\node at (2,0) (63) {};
\node at (1,0.5) (62) {};
\node at (1,1.8) (1) {};
\draw (3)--(1)  node[midway,below left, opacity=0, text opacity=1]  {\scriptsize $\infty\,$};
\draw[dashed] (0)--(1) node[midway,below=7pt, left=-3pt, opacity=0, text opacity=1]%
{\scriptsize $\frac{\sqrt{2}}{2}+\frac{2\sqrt{6}}{3}$};
\draw[dashed] (3)--(4) node[midway,above=6pt, left=-2pt, opacity=0, text opacity=1]%
{\scriptsize $\cos(\frac{\pi}{5})+\frac{2\sqrt{6}}{3}$};
\draw (4)--(61) ;
\draw (0)--(63) ;
\draw (61)--(63);
\draw (61)--(62);
\draw (62)--(63);
%
%
\draw (1)--(4) node[midway,above=6pt, left=-4pt, opacity=0, text opacity=1]  {\scriptsize $4$};
\draw (0)--(3) node[midway,above right, opacity=0, text opacity=1]  {\scriptsize $\,5$};
\draw (0)--(4);
%
\end{tikzpicture}
\caption{\cite[p. 11 (e)]{Roberts}}
\end{centering}\end{subfigure}
\end{figure}

\begin{figure}
\caption{Five examples of Coxeter polytopes in $\Hy^5$ with trace field of degree 4 over $\Q$.} \label{fig:deg4}
\begin{subfigure}[h]{.45\textwidth}\begin{centering}
\begin{tikzpicture}[scale=1.3] 
\tikzstyle{every node}=[draw,shape=circle,minimum size=\minsize,inner sep=0, fill=black]
\node at (0,0) (sw) {};
\node  at (2,0) (se) {};
\node  at (0,1) (w) {};
\node  at (2,1) (e) {};
\node  at (1,2) (nw) {};
\node  at (1,-0.5) (s) {};
\node  at (1,0.5) (n) {};
\node  at (1,1.5) (ne) {};

\draw (ne)--(nw) node[midway,right=5pt, below=-2pt,opacity=0, text opacity=1] {\scriptsize$\infty$};

\draw (w)--(sw);
\draw[dashed] (ne)--(e) node[midway, below=9pt, left= -6pt, opacity=0, text opacity=1]
{\scriptsize$\,\frac{\sqrt{2}}{2}+\frac{\sqrt{6}}{2}$};
\draw[dashed] (nw)--(w) node[midway,above left, opacity=0, text opacity=1]  {\scriptsize$\,\frac{1}{2}+\frac{\sqrt{6}}{2}$};
\draw (se)--(e);
\draw (sw)--(n);
\draw (sw)--(s);
\draw (se)--(n);
\draw (se)--(s);
\draw (nw)--(e);

\draw (0,1)--(1,1.5) node[midway,below right, opacity=0, text opacity=1] {\scriptsize$4$};
%
%
\end{tikzpicture}
\caption{\cite[p. 14 (c)]{Roberts}}
\end{centering}\end{subfigure}
\begin{subfigure}[h]{.45\textwidth}\begin{centering}
\begin{tikzpicture}[scale=1.3]
\tikzstyle{every node}=[draw,shape=circle,minimum size=\minsize,inner sep=0, fill=black]
\node at (0,0) (sw) {};
\node  at (2,0) (se) {};
\node  at (0,1) (w) {};
\node  at (2,1) (e) {};
\node  at (1,2) (nw) {};
\node  at (1,-0.5) (s) {};
\node  at (1,0) (n) {};
\node  at (1,1.5) (ne) {};

\draw (ne)--(nw) node[midway,right=5pt, below=-2pt, opacity=0, text opacity=1] {\scriptsize$\infty$};

\draw (w)--(sw);
\draw[dashed] (ne)--(e) node[midway, below=9pt, left= -6pt, opacity=0, text opacity=1]
{\scriptsize$\,\frac{\sqrt{2}}{2}+\frac{\sqrt{10}}{2}$};
\draw[dashed] (nw)--(w) node[midway,above left, opacity=0, text opacity=1]  {\scriptsize$\,\frac{1}{2}+\frac{\sqrt{10}}{2}$};
\draw (se)--(e);
\draw (sw)--(n);
\draw (sw)--(s);
\draw (se)--(n);
\draw (se)--(s);
\draw (nw)--(e);

\draw (0,1)--(1,1.5) node[midway,below right, opacity=0, text opacity=1] {\scriptsize$4$};

\draw (w) .. controls (1,0.35) .. (e);

\end{tikzpicture}
\caption{\cite[p. 14 (d)]{Roberts}}
\end{centering}\end{subfigure}

\begin{subfigure}[h]{.45\textwidth}\begin{centering}
\begin{tikzpicture}[scale=1.3]
\tikzstyle{every node}=[draw,shape=circle,minimum size=\minsize,inner sep=0, fill=black]
\node  at (0,0) (sw) {};
\node  at (2,0) (se) {};
\node  at (0,1) (w) {};
\node  at (2,1) (e) {};
\node  at (1,2.5) (nw) {};
\node  at (1,-0.25) (s) {};
\node  at (1,0.25) (n) {};
\node  at (1,1.5) (ne) {};

\draw (ne)--(nw) node[midway,right, opacity=0, text opacity=1] {\scriptsize$\infty$};

\draw (w)--(sw);
\draw[dashed] (ne)--(e) node[midway,below=9pt, left= -6pt, opacity=0, text opacity=1]
{\scriptsize$\,\frac{\sqrt{10}}{2}+\frac{1}{2}$};
\draw[dashed] (nw)--(w) node[midway,above left, opacity=0, text opacity=1]
{\scriptsize$\,\frac{\sqrt{2}}{2}+\frac{\sqrt{10}}{2}$};
\draw (se)--(e);

\draw (se)--(s);
\draw (sw)--(s);

\draw (1,-0.25)--(1,0.25) node[midway,right, opacity=0, text opacity=1] {\scriptsize$4$};

\draw (w)--(ne);
\draw (w) .. controls (1,0.35) .. (e);

\draw (1,2.5)--(2,1) node[midway,above right, opacity=0, text opacity=1] {\scriptsize$4$};

\end{tikzpicture}
\caption{\cite[p. 17 (c)]{Roberts}}
\end{centering}\end{subfigure}
\begin{subfigure}[h!]{.45\textwidth}\begin{centering}
\begin{tikzpicture}[scale=1.3]
\tikzstyle{every node}=[draw,shape=circle,minimum size=\minsize,inner sep=0, fill=black]
\node at (0,0) (sw) {};
\node  at (2,0) (se) {};
\node at (0,1) (w) {};
\node at (2,1) (e) {};
\node at (1,2.5) (nw) {};
\node at (1,0) (s) {};
\node at (1,0.5) (n) {};
\node at (1,1.5) (ne) {};

\draw (ne)--(nw) node[midway,right, opacity=0, text opacity=1] {\scriptsize$\infty$};

\draw (w)--(sw);
\draw[dashed] (ne)--(e) node[midway, below=9pt, left= -6pt, opacity=0, text opacity=1]
{\scriptsize$\,\frac{\sqrt{6}}{2}+\frac{1}{2}$};
\draw[dashed] (nw)--(w) node[midway,above left, opacity=0, text opacity=1]  {\scriptsize$\,\frac{\sqrt{2}}{2}+\frac{\sqrt{6}}{2}$};
\draw (se)--(e);

\draw (se)--(sw);

\draw (1,0)--(1,0.5) node[midway,right, opacity=0, text opacity=1] {\scriptsize$4$};

\draw (w)--(ne);



\draw (1,2.5)--(2,1) node[midway,above right, opacity=0, text opacity=1] {\scriptsize$4$};

\end{tikzpicture}
\caption{\cite[p. 17 (d)]{Roberts}}
\end{centering}\end{subfigure}


\begin{subfigure}[h!]{.45\textwidth}\begin{centering}
\begin{tikzpicture}[scale=1.3]
\tikzstyle{every node}=[draw,shape=circle,minimum size=\minsize,inner sep=0, fill=black]
\node at (0,0.5) (61) {};
\node at (2,0.5) (64) {};
\node at (0,1) (4) {};
\node at (2,1) (0) {};
\node at (1,2.5) (3) {};
\node at (0,0) (62) {};
\node at (2,0) (63) {};
\node at (1,1.5) (1) {};

\draw (1)--(3) node[midway,right, opacity=0, text opacity=1] {\scriptsize$\infty$};
\draw (4)--(61);
\draw (0)--(64);
\draw[dashed] (3)--(4) node[midway,above left, opacity=0, text opacity=1]  {\scriptsize$\,\frac{\sqrt{2}}{2}+\frac{\sqrt{26}}{4}$};
\draw[dashed] (1)--(0) node[midway,below=9pt, left= -6pt, opacity=0, text opacity=1]
{\scriptsize$\,\frac{1}{2}+\frac{\sqrt{26}}{4}$};

\draw (64)--(61);
\draw (64)--(63);

\draw (1)--(4);

\draw (1,2.5)--(2,1) node[midway,above right, opacity=0, text opacity=1] {\scriptsize$4$};

\draw (62)--(61);
\draw (62)--(63);

\end{tikzpicture}
\caption{\cite[p. 19 (e)]{Roberts}}
\end{centering}\end{subfigure}
\end{figure}
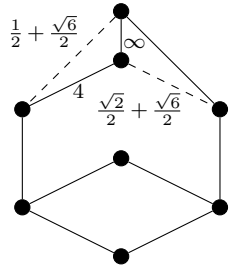
